\documentclass[12pt]{amsart}

\usepackage[T1]{fontenc}
\usepackage{lmodern}
\usepackage{amsmath, amssymb, amsthm}
\usepackage{geometry}
\usepackage{hyperref}
\usepackage{enumerate}
\usepackage{tikz}
\usepackage{pgfplots}
\pgfplotsset{compat=1.18}
\usepackage{booktabs}
\usepackage{comment}
\usepackage{float}

\newtheorem{theorem}{Theorem}[section]
\newtheorem{prop}[theorem]{Proposition}
\newtheorem{lemma}[theorem]{Lemma}

\newtheorem{remark}[theorem]{Remark}
\newtheorem{corollary}[theorem]{Corollary}

\begin{document}

\title[Triangles Associated with $D(4)$-triples]{Euclidean Triangles Associated with Diophantine $D(4)$-triples}

\author[M. Bliznac Trebje\v{s}anin]{Marija {Bliznac Trebje\v{s}anin}}
\address{University of Split, Faculty of Science, Ru\dj{}era Bo\v{s}kovi\'{c}a 33,
21000 Split, Croatia}
\email{marbli@pmfst.hr}

\author[J. Ple\v{s}tina]{Jelena {Ple\v{s}tina}*}
\thanks{*Corresponding author.}
\address{University of Split, Faculty of Science, Ru\dj{}era Bo\v{s}kovi\'{c}a 33,
21000 Split, Croatia}
\email{jplestina@pmfst.hr}

\begin{abstract}
In this paper we study Euclidean triangles associated with Diophantine $D(4)$-triples. For a $D(4)$-triple $\{a,b,c\}$, $a<b<c$, we prove that the lengths
$r=\sqrt{ab+4}$, $s=\sqrt{ac+4}$ and $t=\sqrt{bc+4}$ associated with this $D(4)$-triple form a non-degenerate triangle if and only if the triple is regular and $r>b-a$. For a fixed shortest side $r$, we describe all triangles associated with $D(4)$-triples in terms of the factorisations of $r^2-4$, give a divisor criterion for their number, and determine their angle type. We also prove that two similar triangles associated with $D(4)$-triples are necessarily congruent. Finally, for the family of $D(4)$-triples $\{r-2,r+2,4r\}$, we obtain a Pell-equation characterisation of the Heronian triangles associated with these $D(4)$-triples.
\end{abstract}

\maketitle 



\section{Introduction}

A $D(4)$-triple is a set of three distinct positive integers $\{a,b,c\}$ such that the product of any two distinct elements, increased by $4$, is a perfect square. We order its elements as $a<b<c$ and define
\begin{equation}
r=\sqrt{ab+4},\qquad
s=\sqrt{ac+4},\qquad
t=\sqrt{bc+4}.
\label{eq:lengths}
\end{equation}
Then $r,s,t$ are positive integers such that $r<s<t$. We call $r,s,t$ the \emph{lengths associated with the $D(4)$-triple $\{a,b,c\}$}.

We determine when the lengths $r,s,t$ associated with the $D(4)$-triple $\{a,b,c\}$ are the side lengths of a non-degenerate Euclidean triangle. Since $r<s<t$, this is
equivalent to
\begin{equation}
r+s>t.
\label{eq:triangle-ineq}
\end{equation}
When \eqref{eq:triangle-ineq} holds, we call the triangle with side lengths $r,s,t$ the \emph{triangle associated with the $D(4)$-triple $\{a,b,c\}$}.
A triangle with integer side lengths is called an
\emph{integer triangle}. An integer triangle whose area is also an integer is called a \emph{Heronian triangle}. Hence, every triangle associated with a $D(4)$-triple is an integer triangle.
The equality $r+s=t$ gives a degenerate triangle, whereas if $r+s<t$, $r,s,t$ do not form a triangle. All three cases occur for $D(4)$-triples, as shown in Table~\ref{tab:three-cases}. 
\begin{table}[H]
\centering
\begin{tabular}{c c c c}
\hline
$D(4)$-triple & $(r,s,t)$ & Relation & Case \\
\hline
$\{3,4,15\}$   & $(4,7,8)$   & $4+7>8$   & non-degenerate triangle \\
$\{3,4,224\}$  & $(4,26,30)$ & $4+26=30$ & degenerate triangle \\
$\{3,15,224\}$ & $(7,26,58)$ & $7+26<58$ & no triangle \\
\hline
\end{tabular}
\caption{Examples of the three possible cases.}
\label{tab:three-cases}
\end{table}

A set of two distinct positive integers $\{a,b\}$ is called a
\emph{$D(4)$-pair} if $ab+4$ is a perfect square. If $\{a,b\}$ is a $D(4)$-pair and
$r=\sqrt{ab+4}$, then
\begin{equation}
c=a+b+2r
\label{eq:regular}
\end{equation}
is the smallest integer $c>b$ that extends the pair to a $D(4)$-triple \cite{Bliznac2021}. A $D(4)$-triple satisfying
\eqref{eq:regular} is called \emph{regular}
\cite[Definition~2]{BliznacFilipin2019}.
For a regular $D(4)$-triple, direct calculation gives
\begin{equation}
 s=a+r,\qquad t=b+r.
 \label{eq:regular-identities}
\end{equation}

Arithmetic properties and extensions of $D(4)$-triples
have been studied extensively \cite{Bliznac2021, BliznacFilipin2019, DujellaRamasamy2005,Filipin2009}.

We prove that \eqref{eq:triangle-ineq} holds exactly for regular
$D(4)$-triples satisfying $r>b-a$. We then describe the triangles
associated with such $D(4)$-triples through the identity
$$
(s-r)(t-r)=r^2-4.
$$
For a fixed $r$, the possible side lengths $s$ and $t$ are determined
by suitable factor pairs of $r^2-4$. In particular, the factorisation $r^2-4=(r-2)(r+2)$
satisfies the condition $b-a<r$ precisely for $r\geq5$ and gives the $D(4)$-triple $\{r-2,r+2,4r\}$. 
The family $\{r-2,r+2,4r\}$ is already known in the theory of
$D(4)$-tuples. For every $r\geq3$, $\{r-2,r+2,4r,4r(r^2-1)\}$ is a $D(4)$-quadruple \cite{Dujella1993}. Its first three
elements form the regular $D(4)$-triple $\{r-2,r+2,4r\}$. It is proved that $4r(r^2-1)$ is the unique positive integer $d$ for which
$\{r-2,r+2,4r,d\}$ is a $D(4)$-quadruple \cite{Fujita2006}.
For $r\geq5$, the triangle associated with this $D(4)$-triple has side lengths $(r,2r-2,2r+2)$, so its two longer sides differ by $4$. Sastry \cite{Sastry2001} considered the more general problem of Heronian triangles in which two sides differ by a prescribed integer and related this problem to Pell-type equations. 
Heronian triangles with side lengths $(r,2r-2,2r+2)$ are recorded in OEIS entry A272365 \cite{Jackson2016}, whose terms give their longest side lengths. We give a self-contained proof of the characterisation of all such triangles using a Pell equation.

\section{Preliminaries}\label{sec:preliminaries}

We recall some properties of $D(4)$-triples that will be used in the
proofs. Let $\{a,b,c\}$, $a<b<c$, be a $D(4)$-triple.
Following \cite{DujellaRamasamy2005, Filipin2009}, we define
\begin{equation}
d_{\pm}=d_{\pm}(a,b,c)
=a+b+c+\frac{abc\pm rst}{2},
\label{eq:dpm}
\end{equation}
where $r,s,t$ are defined by \eqref{eq:lengths}.

\begin{remark}[see {\cite{Filipin2009}}] \label{res:dpm-extensions}
For any $D(4)$-triple $\{a,b,c\}$, the set
$\{a,b,c,d_+\}$ is a $D(4)$-quadruple.
Furthermore, if $d_-\neq0$, then
$\{a,b,c,d_-\}$ is also a $D(4)$-quadruple and $d_-<c$.
\end{remark}

We use the following characterisation of regular triples.
\begin{prop}[{\cite[Proposition~2]{BliznacFilipin2019}}]
\label{prop:regular-dminus}
The $D(4)$-triple $\{a,b,c\}$ is a regular triple if and only if
$d_-(a,b,c)=0$.
\end{prop}

\begin{prop}[{\cite[Proposition~3]{BliznacFilipin2019}}]
\label{prop:dminus-inversion}
Let $\{a,b,c\}$ be a $D(4)$-triple, such that $a<b<c$. We have
$a=d_-\bigl(b,c,d_+(a,b,c)\bigr)$,
$b=d_-\bigl(a,c,d_+(a,b,c)\bigr)$,
$c=d_-\bigl(a,b,d_+(a,b,c)\bigr)$.
Moreover, if $\{a,b,c\}$ is not a regular triple, then
$c=d_+\bigl(a,b,$ $d_-(a,b,c)\bigr)$.
In particular $\{a,b,d_-(a,b,c),c\}$ is a regular $D(4)$-quadruple.
\end{prop}

We will use the following direct consequence of \cite[Lemma~3]{Filipin2009}.

\begin{lemma}\label{lem:known-properties}
Let $\{a,b,c\}$ be a $D(4)$-triple with $a<b<c$, and let $r=\sqrt{ab+4}$. Then $c\geq a+b+2r$.
\end{lemma}

\begin{proof}
By \cite[Lemma~3]{Filipin2009}, either $c=a+b+2r$ or $c\geq4b$.
Since $a<b$ and $r^2=ab+4$, we have $r\leq b$, and therefore
$a+b+2r<4b$. Hence $c\geq a+b+2r$.
\end{proof}

\section{The triangle criterion}\label{sec:criterion}

We first establish an inequality for the regular extension of a
$D(4)$-triple. It will be used to exclude non-regular triples in the
proof of the triangle criterion.

\begin{lemma}\label{lem:extension-ineq}
Let $\{p,q,u\}$, $p<q<u$, be a $D(4)$-triple, and let
$$
w=d_+(p,q,u).
$$
Then
\begin{equation}
\sqrt{uw+4}-\sqrt{qw+4}>\sqrt{qu+4}.
\label{eq:lemma-ineq}
\end{equation}
\end{lemma}

\begin{proof}
By Lemma~\ref{lem:known-properties} we have
\begin{equation}
u\ge p+q+2\sqrt{pq+4}
\label{eq:uq-bound}
\end{equation}
and since
$$
pu+4\ge p\bigl(p+q+2\sqrt{pq+4}\bigr)+4
       =\bigl(p+\sqrt{pq+4}\bigr)^2,
$$
we have 
\begin{equation}
\sqrt{pu+4}\ge p+\sqrt{pq+4}.
\label{eq:pu-bound}
\end{equation}
From \eqref{eq:dpm}, since $w=d_+(p,q,u)$, we have
$$
2\sqrt{uw+4}
=u\sqrt{pq+4}+\sqrt{pu+4}\sqrt{qu+4}
$$
and
$$
2\sqrt{qw+4}
=q\sqrt{pu+4}+\sqrt{pq+4}\sqrt{qu+4}.
$$
Hence
\begin{align}
&2\bigl(\sqrt{uw+4}-\sqrt{qw+4}-\sqrt{qu+4}\bigr) \notag\\
&\quad=\sqrt{pq+4}(u-q)
 +\bigl(\sqrt{pu+4}-\sqrt{pq+4}\bigr)
  \bigl(\sqrt{qu+4}-q\bigr)
 -2\sqrt{qu+4}.
\label{eq:keyidentity}
\end{align}
Since $u>q$, we have $\sqrt{qu+4}>q$. Moreover,
\eqref{eq:pu-bound} gives
$$
\sqrt{pu+4}-\sqrt{pq+4}\ge p.
$$
Therefore, \eqref{eq:keyidentity} gives
\begin{align}
&2\bigl(\sqrt{uw+4}-\sqrt{qw+4}-\sqrt{qu+4}\bigr) \notag\\
&\quad\ge \sqrt{pq+4}(u-q)
+p\bigl(\sqrt{qu+4}-q\bigr)-2\sqrt{qu+4} \notag\\
&\quad=
\sqrt{pq+4}(u-q)-pq+(p-2)\sqrt{qu+4}.
\label{eq:keybound}
\end{align}

If $p\ge2$, then \eqref{eq:uq-bound} gives
$$
\sqrt{pq+4}(u-q)
\ge p\sqrt{pq+4}+2pq+8>2pq
$$
and the right-hand side of \eqref{eq:keybound} is
strictly greater than $2pq-pq=pq>0$.

Now let $p=1$. From \eqref{eq:uq-bound}, we have
\begin{equation}
u-q\ge1+2\sqrt{q+4}.
\label{eq:p1lower}
\end{equation}
Hence
$$
\sqrt{q+4}(u-q)-q>0.
$$
Using \eqref{eq:p1lower}, we have
$$
\begin{aligned}
&(q+4)(u-q)-q\bigl(1+2\sqrt{q+4}\bigr)\\
&\qquad\ge
(q+4)\bigl(1+2\sqrt{q+4}\bigr)
-q\bigl(1+2\sqrt{q+4}\bigr)\\
&\qquad=
4\bigl(1+2\sqrt{q+4}\bigr),
\end{aligned}
$$
and consequently,
\begin{align*}
&\bigl(\sqrt{q+4}(u-q)-q\bigr)^2-(qu+4)\\
&\quad=(u-q)\Bigl((q+4)(u-q)
-q\bigl(1+2\sqrt{q+4}\bigr)\Bigr)-4\\
&\quad\ge4(u-q)\bigl(1+2\sqrt{q+4}\bigr)-4>0.
\end{align*}
Therefore
$$
\sqrt{q+4}(u-q)-q>\sqrt{qu+4},
$$
and the right-hand side of \eqref{eq:keybound} is again positive.
\end{proof}

We can now prove the main criterion.

\begin{theorem}
\label{thm:triangle}
Let $\{a,b,c\}$, $a<b<c$, be a $D(4)$-triple, and let $r,s,t$ be
defined by \eqref{eq:lengths}. Then $r,s,t$ are the side lengths of the non-degenerate triangle associated with $\{a,b,c\}$ if and only if $\{a,b,c\}$ is regular and
\begin{equation}
r>b-a.
\label{eq:trianglecriterion}
\end{equation}
\end{theorem}

\begin{proof}
Suppose first that $\{a,b,c\}$ is regular. By
\eqref{eq:regular-identities},
$$
r+s-t=r+a-b.
$$
Hence $r+s>t$ if and only if $r>b-a$.

Now suppose that $\{a,b,c\}$ is non-regular and put $e=d_-(a,b,c)$.
By Proposition~\ref{prop:regular-dminus}, we have $e\neq0$.
By Remark~\ref{res:dpm-extensions}, $\{a,b,c,e\}$ is a $D(4)$-quadruple and $e<c$.
Therefore $0<e<c$, $e\notin\{a,b\}$, and $\{a,b,e\}$ is a $D(4)$-triple.
By Proposition~\ref{prop:dminus-inversion}, we have $c=d_+(a,b,e)$.
We distinguish the three cases for $e$.

If $e<a$, then $\{e,a,b\}$ is a $D(4)$-triple and
$c=d_+(e,a,b)$. Applying Lemma~\ref{lem:extension-ineq} with
$(p,q,u,w)=(e,a,b,c)$ gives
$$
\sqrt{bc+4}-\sqrt{ac+4}>\sqrt{ab+4},
$$
that is, $t-s>r$. Hence $r+s<t$.

If $a<e<b$, then $\{a,e,b\}$ is a $D(4)$-triple and
$c=d_+(a,e,b)$. Applying Lemma~\ref{lem:extension-ineq} with
$(p,q,u,w)=(a,e,b,c)$ gives
$$
t-\sqrt{ec+4}>\sqrt{eb+4}.
$$
Since $e>a$,
$$
\sqrt{ec+4}>s,
\qquad
\sqrt{eb+4}>r,
$$
so $t>s+r$.

It remains to consider $b<e<c$. Since $\{a,b,e\}$ is a $D(4)$-triple,
Lemma~\ref{lem:known-properties} gives
\begin{equation}
e\ge a+b+2r.
\label{eq:e-min}
\end{equation}
Since $c=d_+(a,b,e)$, a direct calculation from \eqref{eq:dpm} gives
\begin{align*}
2s&=a\sqrt{be+4}+r\sqrt{ae+4},\\
2t&=b\sqrt{ae+4}+r\sqrt{be+4}.
\end{align*}
Therefore
\begin{equation}
2(t-s)
=(b-r)\sqrt{ae+4}+(r-a)\sqrt{be+4}.
\label{eq:tsdiff}
\end{equation}

Since $r^2=ab+4>a^2$, we have $r>a$, and the proof of
Lemma~\ref{lem:known-properties} gives $r\le b$. Hence
\begin{equation}
a<r\le b.
\label{eq:rbetween}
\end{equation}

Thus the coefficients in \eqref{eq:tsdiff} are nonnegative. From
\eqref{eq:e-min},
\begin{align*}
2(t-s)
&\ge (b-r)\sqrt{a(a+b+2r)+4}
 +(r-a)\sqrt{b(a+b+2r)+4}\\
&=(b-r)(a+r)+(r-a)(b+r)\\
&=2r(b-a).
\end{align*}
Here we used $r^2=ab+4$, since $a(a+b+2r)+4=(a+r)^2$
and $b(a+b+2r)+4=(b+r)^2$.
Since $a<b$ are integers, $b-a\ge1$, and therefore $2(t-s)\ge2r$.
Hence $r+s\le t$, so the non-degenerate triangle associated with
$\{a,b,c\}$ does not exist.
\end{proof}

As an immediate consequence of Theorem~\ref{thm:triangle}, every
triangle associated with a $D(4)$-triple has shortest side at least $4$.
Since $a$ and $b$ are distinct positive integers, $ab\ge2$, and therefore $r^2=ab+4\ge6.$ Hence $r\ge3$. 
If $r=3$, then $ab=5,$ so $(a,b)=(1,5)$. In this case $b-a=4>3=r$, and Theorem~\ref{thm:triangle}
shows that the associated triangle does not exist. Therefore every triangle associated with a $D(4)$-triple has shortest side $r\ge4$. The bound is attained for the $D(4)$-triple $\{3,4,15\}$, whose associated triangle has side lengths $(4,7,8)$.

\begin{remark}
Recall that a set of distinct positive integers has the property $D(n)$
if the product of any two distinct elements, increased by $n$, is a
perfect square.
\begin{enumerate}[(i)]

\item
Let $n$ be a positive integer, let $\{a,b\}$ be a $D(n)$-pair with $a<b$, and put $r=\sqrt{ab+n}$. For the regular extension
$c=a+b+2r$, we have $ac+n=(a+r)^2$ and $bc+n=(b+r)^2$ (see, for example, \cite{Dujella1996}).
Hence, if $s=\sqrt{ac+n}$ and $t=\sqrt{bc+n}$, then
$s=a+r$ and $t=b+r$. Therefore, the associated lengths $r,s,t$ form a non-degenerate triangle if and only if $r>b-a$.

\item
For $n=4$, Theorem~\ref{thm:triangle} shows that the converse also holds: if the lengths associated with a $D(4)$-triple form a non-degenerate triangle, then the triple is necessarily regular. This implication does not hold for arbitrary $n$. For example, $\{5,8,104\}$ is a non-regular $D(9)$-triple with corresponding lengths $(7,23,29)$, and $7+23>29$.

\item
Theorem~\ref{thm:triangle} also gives a criterion for $D(1)$-triples. Let $\{a,b,c\}$, $a<b<c$, be a $D(1)$-triple, and put
$$
r=\sqrt{ab+1},\qquad
s=\sqrt{ac+1},\qquad
t=\sqrt{bc+1}.
$$
Then $\{2a,2b,2c\}$ is a $D(4)$-triple whose associated lengths are $2r,2s,2t$. By Theorem~\ref{thm:triangle}, $r+s>t$ if and only if $2c=2a+2b+4r$ and $2r>2b-2a$. Equivalently, $r,s,t$ form a non-degenerate triangle if and only if $c=a+b+2r$ and $r>b-a$.
\end{enumerate}
\end{remark}

The triangle criterion also admits an equivalent formulation in terms
of $a$ and $b$, and gives a direct expression for the area.

\begin{corollary}
Let $\{a,b,c\}$, $a<b<c$, be a $D(4)$-triple, and let $r,s,t$ be
defined by \eqref{eq:lengths}. Then the triangle associated with the $D(4)$-triple $\{a,b,c\}$ exists if and only if $\{a,b,c\}$ is regular and
\begin{equation}
a^2+b^2<3ab+4.
\label{eq:arithcriterion}
\end{equation}
If the triangle associated with the $D(4)$-triple $\{a,b,c\}$ exists, its area $A$ satisfies
\begin{equation}
16A^2=(c^2-r^2)\bigl(3ab+4-a^2-b^2\bigr).
\label{eq:area}
\end{equation}
\end{corollary}

\begin{proof}
Theorem~\ref{thm:triangle} shows that the triangle associated with the
$D(4)$-triple $\{a,b,c\}$ exists if and
only if $\{a,b,c\}$ is regular and $r>b-a$. Since $r^2=ab+4$,
$$
r>b-a
\quad\Longleftrightarrow\quad
ab+4>(b-a)^2,
$$
which is equivalent to \eqref{eq:arithcriterion}.

If the triangle associated with the $D(4)$-triple $\{a,b,c\}$ exists, then $\{a,b,c\}$ is regular. Hence
\eqref{eq:regular} and \eqref{eq:regular-identities} hold and Heron's formula gives
\begin{align*}
16A^2
&=(r+s+t)(-r+s+t)(r-s+t)(r+s-t)\\
&=(a+b+3r)(a+b+r)(r+b-a)(r-b+a)\\
&=(c^2-r^2)\bigl(r^2-(b-a)^2\bigr)\\
&=(c^2-r^2)\bigl(3ab+4-a^2-b^2\bigr).
\end{align*}
\end{proof}

\section{Properties of triangles associated with $D(4)$-triples}\label{sec:geometry}

We now derive some geometric consequences of
Theorem~\ref{thm:triangle}.

By Theorem~\ref{thm:triangle}, whenever the triangle with lengths $\{r,s,t\}$ associated with a
$D(4)$-triple $\{a,b,c\}$ exists, the triple is regular. Hence
\eqref{eq:regular-identities} gives $s-r=a$ and $t-r=b$. Since
$r^2=ab+4$, it follows that $(s-r)(t-r)=r^2-4.$
We formulate this as a lemma for later reference.

\begin{lemma} \label{lem:associated-identities}
Let $\{a,b,c\}$, $a<b<c$, be a $D(4)$-triple, and let $r,s,t$ be
defined by \eqref{eq:lengths}. If the triangle associated with the $D(4)$-triple $\{a,b,c\}$ exists, then
\begin{equation}
(s-r)(t-r)=r^2-4.
\label{eq:factorised}
\end{equation}
\end{lemma}

\subsection{Fixed shortest side}

\begin{corollary}\label{cor:factor-param}
For an integer $r\ge4$, the side-length triples $(r,s,t)$ of triangles associated with $D(4)$-triples are exactly
\begin{equation}
(r,s,t)=(r,r+a,r+b),
\label{eq:factor-param}
\end{equation}
where $a,b$ are positive integers satisfying
\begin{equation}
r^2-4=ab,\qquad 0<a<b,\qquad b-a<r.
\label{eq:factor-pairs}
\end{equation}
The corresponding $D(4)$-triple is uniquely determined and is
\begin{equation}
\{a,b,a+b+2r\}.
\label{eq:triple-from-factors}
\end{equation}
\end{corollary}

\begin{proof}
Let $(r,s,t)$ be the side-length triple of the triangle associated with the $D(4)$-triple $\{a,b,c\}$. By Theorem~\ref{thm:triangle}, the
$D(4)$-triple is regular and $b-a<r$. Hence
\eqref{eq:regular-identities} holds, so $s=r+a, t=r+b.$
From \eqref{eq:lengths}, we also have $ab=r^2-4$.
Thus \eqref{eq:factor-param} and \eqref{eq:factor-pairs} hold.

Conversely, let $a,b$ satisfy \eqref{eq:factor-pairs} and put $c=a+b+2r$. Then $0<a<b<c,$ and
$$
ab+4=r^2,
\qquad
ac+4=(a+r)^2,
\qquad
bc+4=(b+r)^2.
$$
Hence $\{a,b,c\}$ is a regular $D(4)$-triple, and its corresponding
values in \eqref{eq:lengths} are
$ r, r+a, r+b$.
Since $b-a<r$, Theorem~\ref{thm:triangle} shows that these are the side lengths of the triangle associated with the $D(4)$-triple
$\{a,b,c\}$.
Finally, any $D(4)$-triple with side-length triple $(r,s,t)$ must be
regular by Theorem~\ref{thm:triangle}. Therefore $a=s-r, b=t-r,$ and $c=a+b+2r=s+t$. 
\end{proof}

\begin{corollary}
\label{cor:angle-type}
Let $\{a,b,c\}$, $a<b<c$, be a $D(4)$-triple, and let $r,s,t$ be
defined by \eqref{eq:lengths}. Suppose that $r,s,t$ are the side lengths
of a non-degenerate triangle.
Then
\begin{equation}
r^2+s^2-t^2
=r^2-(b-a)(2r+a+b)
\label{eq:anglecriterion}
\end{equation}
and exactly one of the following three cases holds for the triangle associated with the $D(4)$-triple $\{a,b,c\}$: 
\begin{enumerate}[(1)]
\item it is acute if and only if
$$
r^2>(b-a)(2r+a+b).
$$
\item it is right if and only if
$$
r^2=(b-a)(2r+a+b).
$$
\item it is obtuse if and only if
$$
r^2<(b-a)(2r+a+b).
$$
\end{enumerate}
\end{corollary}

\begin{proof}
By Theorem~\ref{thm:triangle}, the $D(4)$-triple is regular. Hence
\eqref{eq:regular-identities} holds, that is,
$$
s=a+r,\qquad t=b+r.
$$
Therefore
$$
\begin{aligned}
r^2+s^2-t^2
&=r^2+(r+a)^2-(r+b)^2\\
&=r^2-(b-a)(2r+a+b).
\end{aligned}
$$
Since $t$ is the longest side, the three cases follow from the sign of
$r^2+s^2-t^2$.
\end{proof}

We next use the factorisation in Corollary~\ref{cor:factor-param} to
count the triangles with a prescribed shortest side.

\begin{corollary}
\label{cor:count-shortest}
Let $r\ge4$ be a fixed integer. The number of pairwise non-congruent triangles associated with $D(4)$-triples and having shortest side $r$ equals the number of positive divisors $a$ of $r^2-4$ satisfying
\begin{equation}
\frac{\sqrt{5r^2-16}-r}{2}
<a<
\sqrt{r^2-4}.
\label{eq:count-shortest}
\end{equation}
\end{corollary}

\begin{proof}
By Corollary~\ref{cor:factor-param}, the associated triangles with shortest side $r$ correspond to factor pairs $r^2-4=ab$ satisfying $0<a<b$ and
$b-a<r$. Since $b=(r^2-4)/a$, the condition $a<b$ is equivalent to
$a<\sqrt{r^2-4}$. The condition $b-a<r$ is equivalent to
$a^2+ra-(r^2-4)>0$, and hence to $a>(\sqrt{5r^2-16}-r)/2$.
Distinct admissible divisors $a$ give distinct ordered side-length
triples $\left(r,r+a,r+\frac{r^2-4}{a}\right),$
and hence distinct congruence classes.
\end{proof}

Table~\ref{tab:examples} lists the triangles associated with
$D(4)$-triples for $4\le r\le20$.

\begin{table}[H]
\centering
\small
\caption{Triangles associated with $D(4)$-triples for
$4\le r\le20$.}
\label{tab:examples}
\begin{tabular}{cccc}
\toprule
$r$ & $D(4)$-triple & triangle side lengths & angle type \\
\midrule
4 & $\{3,4,15\}$ & $(4,7,8)$ & acute \\
5 & $\{3,7,20\}$ & $(5,8,12)$ & obtuse \\
6 & $\{4,8,24\}$ & $(6,10,14)$ & obtuse \\
7 & $\{5,9,28\}$ & $(7,12,16)$ & obtuse \\
8 & $\{5,12,33\}$ & $(8,13,20)$ & obtuse \\
  & $\{6,10,32\}$ & $(8,14,18)$ & obtuse \\
9 & $\{7,11,36\}$ & $(9,16,20)$ & obtuse \\
10 & $\{8,12,40\}$ & $(10,18,22)$ & obtuse \\
11 & $\{9,13,44\}$ & $(11,20,24)$ & obtuse \\
12 & $\{10,14,48\}$ & $(12,22,26)$ & obtuse \\
13 & $\{11,15,52\}$ & $(13,24,28)$ & obtuse \\
14 & $\{12,16,56\}$ & $(14,26,30)$ & obtuse \\
15 & $\{13,17,60\}$ & $(15,28,32)$ & obtuse \\
16 & $\{12,21,65\}$ & $(16,28,37)$ & obtuse \\
   & $\{14,18,64\}$ & $(16,30,34)$ & right \\
17 & $\{15,19,68\}$ & $(17,32,36)$ & acute \\
18 & $\{16,20,72\}$ & $(18,34,38)$ & acute \\
19 & $\{17,21,76\}$ & $(19,36,40)$ & acute \\
20 & $\{18,22,80\}$ & $(20,38,42)$ & acute \\
\bottomrule
\end{tabular}
\end{table}

For $r=8$, we have $r^2-4=60$. The six positive factor pairs $ab=60$ with $a<b$ determine the points $(s,t)=(8+a,8+b)$ such that $(s-8)(t-8)=60$. They are listed in Table~\ref{tab:r8-factors}. Only the pairs $(5,12)$ and $(6,10)$ satisfy $b-a<8$, and therefore give the triangles $(8,13,20)$ and $(8,14,18)$ associated with the corresponding $D(4)$-triples.

\begin{table}[H]
\centering
\caption{Factor pairs of $60$ and the corresponding points $(s,t)$ for $r=8$.}
\label{tab:r8-factors}
\begin{tabular}{cccc}
\toprule
$(a,b)$ & $(s,t)$ & $b-a<8$ & Associated triangle \\
\midrule
$(1,60)$ & $(9,68)$  & no  & -- \\
$(2,30)$ & $(10,38)$ & no  & -- \\
$(3,20)$ & $(11,28)$ & no  & -- \\
$(4,15)$ & $(12,23)$ & no  & -- \\
$(5,12)$ & $(13,20)$ & yes & $(8,13,20)$ \\
$(6,10)$ & $(14,18)$ & yes & $(8,14,18)$ \\
\bottomrule
\end{tabular}
\end{table}

\subsection{Similarity of triangles associated with $D(4)$-triples}
\begin{prop}\label{prop:similarity}
Let $\{a,b,c\}$ and $\{a',b',c'\}$ be $D(4)$-triples, and let the
triangles associated with these $D(4)$-triples have side-length triples $(r,s,t)$ and $(r',s',t')$, respectively. If the two triangles are similar, then they are congruent.
\end{prop}

\begin{proof}
Since $r<s<t$ and $r'<s'<t'$, similarity gives $(r',s',t')=(\lambda r,\lambda s,\lambda t)$ for some $\lambda>0$. Equation~\eqref{eq:factorised} holds for both triangles. Hence
$$
\lambda^2(s-r)(t-r)
=(s'-r')(t'-r')
=r'^2-4
=\lambda^2r^2-4.
$$
Since $(s-r)(t-r)=r^2-4$, we obtain
$\lambda^2(r^2-4)=\lambda^2r^2-4$, and therefore $\lambda^2=1$.
Thus $\lambda=1$, so the two triangles are congruent.
\end{proof}

\subsection{The family of $D(4)$-triples $\{r-2,r+2,4r\}$}

For $r\ge5$, consider the factor pair
$a=r-2$ and $b=r+2$. Then $ab=r^2-4$ and $b-a=4<r$. Hence Corollary~\ref{cor:factor-param} gives the $D(4)$-triple $\{r-2,r+2,4r\}$, and the triangle associated with this $D(4)$-triple
has side lengths
\begin{equation}
(r,2r-2,2r+2).
\label{eq:fujita-lengths}
\end{equation}

For the triangle associated with the $D(4)$-triple
$\{r-2,r+2,4r\}$, we have $r^2+(2r-2)^2-(2r+2)^2=r(r-16)$. Therefore it is obtuse for $5\le r\le15$, right for $r=16$, and acute for $r\ge17$.

For the triangles associated with the $D(4)$-triples
$\{r-2,r+2,4r\}$, Heron's formula gives
\begin{equation}
16A^2=15r^2(r^2-16).
\label{eq:fujita-area}
\end{equation}

\begin{prop}\label{prop:heronian}
Let $r\ge5$ be an integer. The triangle associated with the $D(4)$-triple $\{r-2,r+2,4r\}$, with side lengths $(r,2r-2,2r+2)$, is Heronian if and only if $r=r_n$ for some $n\ge1$, where
\begin{equation}
r_0=4,\qquad
r_1=16,\qquad
r_{n+1}=8r_n-r_{n-1}\quad(n\ge1).
\label{eq:rrecurrence}
\end{equation}
\end{prop}

\begin{proof}
By \eqref{eq:fujita-area}, the area of the triangle with side lengths $(r,2r-2,2r+2)$ is
$$
A=\frac{r}{4}\sqrt{15(r^2-16)}.
$$
Suppose first that $A$ is an integer. Then $\sqrt{15(r^2-16)}=4A/r$ is rational. Since
$15(r^2-16)$ is an integer, it follows that
$\sqrt{15(r^2-16)}$ is an integer. Let $m=\sqrt{15(r^2-16)}$. Then
\begin{equation}
m^2-15r^2=-240.
\label{eq:pellm}
\end{equation}
Reducing equation \eqref{eq:pellm} modulo $15$ gives $m^2\equiv0\pmod{15}$.
Hence $3\mid m^2$ and $5\mid m^2$, so
$3\mid m$ and $5\mid m$. Therefore $15\mid m$.
Reducing \eqref{eq:pellm} modulo $16$, gives
$
m^2+r^2\equiv0\pmod{16}.
$
Since the quadratic residues modulo $16$ are $0,1,4,$ and $9$, this is
possible only if
$
m^2\equiv r^2\equiv0\pmod{16}.
$
Thus $4\mid m$ and $4\mid r$. Since $\gcd(15,4)=1$, we have $60\mid m$. Hence we may write
$m=60V$ and $r=4U$ for some positive integers $U,V$. Substitution
in \eqref{eq:pellm} gives
\begin{equation}
U^2-15V^2=1.
\label{eq:pell15}
\end{equation}

Conversely, let $(U,V)$ be a positive integer solution of
\eqref{eq:pell15} and put $r=4U$. Then
$15(r^2-16)=240(U^2-1)=(60V)^2$, and therefore
$A=60UV$ is an integer. Thus, for $r\ge5$, the triangle is Heronian
if and only if $r=4U$ for a positive integer solution $(U,V)$ of
\eqref{eq:pell15}.

We now describe all positive integer solutions of \eqref{eq:pell15}.
Let $(U,V)$ be one of them. A direct calculation gives
$$
(4U-15V)^2-15(4V-U)^2=U^2-15V^2=1,
$$
so $(4U-15V,4V-U)$ is again an integer solution of
\eqref{eq:pell15}. Since $U^2=15V^2+1$ and $V\ge1$, we have
$\sqrt{15}\,V<U\le4V$. As $\sqrt{15}>15/4$, it follows that
$4U-15V>0$. Moreover, $U\le4V$ gives $4V-U\ge0$, while
$U>\sqrt{15}\,V>3V$ gives
$
0\le4V-U<V.
$
Hence repeated application of this reduction takes every positive
solution of \eqref{eq:pell15} to $(1,0)$ in finitely many steps.
Reversing the reduction, every nonnegative solution is therefore obtained
from $(1,0)$ by repeatedly applying
$(U,V)\mapsto(4U+15V,U+4V)$. Writing the solution obtained after
$n$ steps as $(U_n,V_n)$, we have
$
U_n+V_n\sqrt{15}=(4+\sqrt{15})^n, n\ge0.
$
Since $4+\sqrt{15}$ and $4-\sqrt{15}$ are the roots of
$x^2-8x+1=0$, the sequence $(U_n)$ satisfies
$U_{n+1}=8U_n-U_{n-1}$, with $U_0=1$ and $U_1=4$. Hence
$r_n=4U_n$ satisfies \eqref{eq:rrecurrence}.

For $n=0$, we have $r_0=4$, which gives the degenerate side-length triple $(4,6,10)$. Therefore, for $r\ge5$, the triangle associated
with the $D(4)$-triple $\{r-2,r+2,4r\}$ is Heronian if and only if
$r=r_n$ for some $n\ge1$.
\end{proof}

The first three Heronian triangles associated with $D(4)$-triples
in the family $\{r-2,r+2,4r\}$ have side lengths $(16,30,34)$,
$(124,246,250)$, and $(976,1950,1954)$, with areas $240$, $14880$,
and $922320$, respectively.

\begin{remark}
For $n\ge1$, the longest side lengths $2r_n+2$ of the Heronian
triangles associated with the $D(4)$-triples in
Proposition~\ref{prop:heronian} are the terms $a(n+1)$ of OEIS entry A272365. Since $r_n$ satisfies a recurrence
with characteristic polynomial $x^2-8x+1$, the sequence $2r_n+2$
satisfies the recurrence with characteristic polynomial
$(x-1)(x^2-8x+1)=x^3-9x^2+9x-1$, that is,
$$
a(n)=9a(n-1)-9a(n-2)+a(n-3),
$$
which is the recurrence given in \cite{Jackson2016}.
Proposition~\ref{prop:heronian} shows that every Heronian triangle
associated with a $D(4)$-triple in the family $\{r-2,r+2,4r\}$
has its longest side length in this sequence.
\end{remark}

\section{Conclusions}
We have shown that the lengths associated with a $D(4)$-triple form a non-degenerate Euclidean triangle precisely when the triple is regular and $r>b-a$. This characterisation reduces the study of triangles associated with $D(4)$-triples and having fixed shortest side $r$ to suitable factorisations of $r^2-4$, yielding their enumeration and angle classification. It also implies that two similar triangles associated with $D(4)$-triples
are necessarily congruent, while for the family of $D(4)$-triples
$\{r-2,r+2,4r\}$ the Heronian condition reduces to a Pell equation. Thus, the geometric properties considered here are determined by the arithmetic structure of the underlying $D(4)$-triple.

\noindent\textbf{Acknowledgements.}
Authors were supported by the University of Split, grant no.\@ IP-UNIST-44, funded by the European Union -- NextGenerationEU.

\end{document}